\documentclass[11pt,reqno]{amsart}

\usepackage{graphicx}
\usepackage[table]{xcolor}
\usepackage{xspace,tikz}
\usepackage{amsmath,amsthm,amsfonts,amssymb,latexsym,mathrsfs,color,extarrows}
\usepackage{hyperref}

\newtheorem{theorem}{Theorem}
\newtheorem{corollary}[theorem]{Corollary}
\newtheorem{proposition}[theorem]{Proposition}

\newtheorem{definition}[theorem]{Definition}
\newtheorem{example}[theorem]{Example}

\newtheorem{problem}[theorem]{Problem}

\newcommand{\des}{{\rm des\,}}

\newcommand{\msn}{\mathfrak{S}_n}

\newcommand{\lrf}[1]{\lfloor #1\rfloor}

\newcommand{\mbn}{{\mathcal B}_n}

\newcommand{\Stirling}[2]{\genfrac{\{}{\}}{0pt}{}{#1}{#2}}

\title[Derivative polynomials]{Recurrences for the derivative polynomials for tangent and secant}
\author{Guo-Niu Han}
\address{I.R.M.A., UMR 7501, Universit\'e de Strasbourg et CNRS, 7 rue Ren\'e Descartes, F-67084 Strasbourg, France}
\email{guoniu.han@unistra.fr (G.-N.~Han)}
\author{Shi-Mei Ma}
\address{School of Mathematics and Statistics,
        Northeastern University at Qinhuangdao,
         Hebei 066000, P.R. China}
\email{shimeimapapers@163.com (S.-M. Ma)}
\begin{document}
\maketitle
\begin{abstract}
In this paper, we choose the derivative polynomials for tangent and secant as basis sets of polynomial space.
From this viewpoint, we first give an expansion of the derivative polynomials for tangent in terms of the derivative polynomials for secant,
and we then present a result in the reverse direction.
We also discuss the relationships between alternating derivative polynomials and Eulerian
polynomials. As applications, we give certain expansions of the
alternating derivative polynomials, which indicate that the
alternating derivative polynomials share more properties with the Chebyshev polynomials.
\bigskip

\end{abstract}
\date{\today}

%
\section{Introduction}
Let $\msn$ be the symmetric group of all permutations of $[n]$, where $[n]=\{1,2,\ldots,n\}$.
Let $\pi=\pi(1)\pi(2)\cdots\pi(n)\in\msn$.
We say that $\pi$ is {\it alternating} if $$\pi(1)>\pi(2)<\pi(3)>\cdots \pi(n).$$
In other words, $\pi(i)<\pi({i+1})$ if $i$ is even and $\pi(i)>\pi({i+1})$ if $i$ is odd.
Let $E_n$ denote the number of alternating permutations in $\msn$.
For instance, $E_4=\#\{2143,3142,3241,4132,4231\}$.
A famous result of Andr\'e~\cite{Andre79} says that
\begin{equation}\label{En}
 \sum_{n=0}^\infty E_n\frac{z^n}{n!}=\tan z+\sec z=1+z+\frac{z^2}{2!}+2\frac{z^3}{3!}+5\frac{z^4}{4!}+16\frac{z^5}{5!}+\cdots.
\end{equation}
Since Euler used~\eqref{En} as the definition of $E_n$, the numbers $E_n$ are called {\it Euler numbers}.
Note that
$$\sum_{n=0}^\infty E_{2n+1}\frac{z^{2n+1}}{(2n+1)!}=\tan z,~~\sum_{n=0}^\infty E_{2n}\frac{z^{2n}}{(2n)!}=\sec z.$$
For this reason the numbers $E_{2n+1}$ are sometimes called {\it tangent numbers} and the numbers $E_{2n}$ are called {\it secant numbers}.

It is clear that
 \begin{equation}\label{Differential00}
	\left\{
	\begin{array}{l}
\frac{\mathrm{d}}{\mathrm{d}\theta}\tan\theta=1+\tan^2\theta,\\
\frac{\mathrm{d}}{\mathrm{d}\theta}\sec\theta=\tan\theta \sec\theta.
	\end{array}
	\right.
	\end{equation}
Besides~\eqref{Differential00}, there are some differential systems can be also generated by trigonometric functions.
As illustrations, we list four examples.
\begin{example}
Setting $y(\theta)=\sec\theta$ and $z(\theta)=\tan\theta$, one has
 \begin{equation*}\label{Differential01}
	\left\{
	\begin{array}{l}
\frac{\mathrm{d}}{\mathrm{d}\theta}y(\theta)=y(\theta)z(\theta),\\
\frac{\mathrm{d}}{\mathrm{d}\theta}z(\theta)=y^2(\theta),
	\end{array}
	\right.
	\end{equation*}
which can be used to generate peak and left peak polynomials (see~\cite{Chen17,Ma121} for details).
\end{example}
\begin{example}
Let $x(\theta)=\left(\tan \theta+\sec \theta\right)^q$, $y(\theta)=\sec\theta$ and $z(\theta)=\tan\theta$, where $q$ is a given parameter. One has
\begin{equation*}\label{Differential02}
	\left\{
	\begin{array}{l}
\frac{\mathrm{d}}{\mathrm{d}\theta}x(\theta)=qx(\theta)y(\theta), \\
\frac{\mathrm{d}}{\mathrm{d}\theta}y(\theta)=y(\theta)z(\theta),\\
\frac{\mathrm{d}}{\mathrm{d}\theta}z(\theta)=y^2(\theta),
	\end{array}
	\right.
	\end{equation*}
which can be used to generate run polynomials of permutations (see~\cite[Section~3]{Ma2001}).
\end{example}
\begin{example}
Setting $x(\theta)=\sec \theta \left(\tan \theta+\sec \theta\right)$ and $y(\theta)=\tan\theta+\sec\theta$, one has
 \begin{equation*}\label{Differential03}
	\left\{
	\begin{array}{l}
\frac{\mathrm{d}}{\mathrm{d}\theta}x(\theta)=x(\theta)y(\theta),\\
\frac{\mathrm{d}}{\mathrm{d}\theta}y(\theta)=x(\theta),
	\end{array}
	\right.
	\end{equation*}
which can be used to generate Andr\'e polynomials of 0-1-2 increasing trees (see~\cite{Chen17,Dumont96}).
\end{example}
\begin{example}
Let $x(\theta)=\sec (2\theta)$ and $y(\theta)=2\tan (2\theta)$. One has
 \begin{equation*}\label{Differential03}
	\left\{
	\begin{array}{l}
\frac{\mathrm{d}}{\mathrm{d}\theta}x(\theta)=x(\theta)y(\theta),\\
\frac{\mathrm{d}}{\mathrm{d}\theta}y(\theta)=4x^2(\theta),
	\end{array}
	\right.
	\end{equation*}
which can be used to generate $\gamma$-coefficient polynomials of the types $A$ and $B$ Eulerian and Narayana polynomials (see~\cite{Ma14} for details).
\end{example}

In this paper, we focus on the derivative polynomials generated by~\eqref{Differential00}.
Note that
$$\frac{\mathrm{d}}{\mathrm{d}\theta}(\tan^n \theta)=n\tan^{n-1} \theta(1+\tan^2 \theta),$$
$$\frac{\mathrm{d}}{\mathrm{d}\theta}(\sec\theta\tan^n \theta)=\sec\theta(n\tan^{n-1}\theta+(n+1)\tan^{n+1}\theta).$$
The {\it derivative polynomials} for tangent and secant are respectively defined as follows:
$$\frac{\mathrm{d}^n}{\mathrm{d}\theta^n}\tan\theta=P_n(\tan \theta),~\frac{\mathrm{d}^n}{\mathrm{d}\theta^n}\sec\theta=\sec\theta~{Q}_n(\tan \theta).$$
The study of these polynomials was initiated by Knuth and Buckholtz~\cite{Knuth}. Recently, there has been much work on various properties and applications of derivative polynomials (see~\cite{Boyadzhiev07,Cvijovic09,Hoffman99,Josuat14,Ma12}).
 As discussed by Adamchik~\cite{Adamchik07}, Boyadzhiev~\cite{Boyadzhiev07}, Hoffman~\cite{Hoffman95,Hoffman99} and Qi~\cite{Qi2015},
 the derivative polynomials can be used to express some improper integrals and infinite series, including
Hurwitz zeta functions, Dirichlet $L$-series and series of powers of reciprocals of integers.

In~\cite{Knuth},
Knuth and Buckholtz noted that $$P_{2n+1}(0)=E_{2n+1},~{Q}_{2n}(0)=E_{2n},$$
and by the chain rule, they deduced that
\begin{equation}\label{DerivativeRecu}
P_{n+1}(x)=(1+x^2)\frac{\mathrm{d}}{\mathrm{d}x}P_n(x),~
{Q}_{n+1}(x)=(1+x^2)\frac{\mathrm{d}}{\mathrm{d}x}{Q}_n(x)+x{Q}_n(x),
\end{equation}
with $P_0(x)=x$ and ${Q}_0(x)=1$.
Carlitz and Scoville~\cite{Carlitz72} derived that
\begin{equation}\label{EGF}
P(x;z)=\sum_{n=0}^\infty P_n(x)\frac{z^n}{n!}=\frac{x+\tan z}{1-x\tan z},~{Q}(x;z)=\sum_{n=0}^\infty {Q}_n(x)\frac{z^n}{n!}=\frac{\sec z}{1-x\tan z}.
\end{equation}
In~\cite{Hoffman95}, Hoffman noted that
\begin{equation}\label{Pn1Qn1}
P_n(1)=2^n(P_n(0)+Q_n(0))=\left\{
                            \begin{array}{ll}
                              2^nQ_n(0)=2^{2k}E_{2k}, & \hbox{if $n=2k$ is even;} \\
                              2^nP_n(0)=2^{2k+1}E_{2k+1}, & \hbox{if $n=2k+1$ is odd.}
                            \end{array}
                          \right.
\end{equation}
Subsequently, Hoffman~\cite{Hoffman99} discussed the other particular values of derivative polynomials.
For example, Hoffman~\cite{Hoffman99} noted that the numbers ${Q}_n(1)$ are the Springer numbers of root systems of type $B_n$, which also count snakes of type $B_n$.
A {\it snake} of type $B_n$ is a signed permutation $\pi(1)\pi(2)\cdots\pi(n)$ of $B_n$ such that $0<\pi(1)>\pi(2)<\cdots\pi(n)$.
Setting $s_n={Q}_n(1)$, one has
\begin{equation*}\label{sn}
\sum_{n=0}^{\infty}s_n\frac{z^n}{n!}=\frac{1}{\cos z-\sin z}.
\end{equation*}
One of the main results of Hoffman is the following expression of $Q_n(1)$ (see~\cite[Theorem~3.1]{Hoffman99}):
\begin{equation}\label{Qn1Pn1}
Q_n(1)=-\sin\frac{n\pi}{2}+\sum_{k=0}^{\lrf{n/2}}\binom{n}{2k}(-1)^kP_{n-2k}(1).
\end{equation}
This paper is motivated by the following problem.
\begin{problem}\label{Problem1}
Are there some generalizations or variations of~\eqref{Qn1Pn1}?
\end{problem}

The Chebyshev polynomials $T_n(x)$ of the first kind are defined by
$$T_n(x)=\cos (n\theta),~~{\text{when $x=\cos(\theta)$}},$$
while the second kind Chebyshev polynomials $U_n(x)$ are defined by
$$U_n(x)=\frac{\sin((n+1)\theta)}{\sin(\theta)},~{\text{when $x=\cos(\theta)$}}.$$
It is well known that the polynomials $T_n(x)$ and $U_n(x)$
are both alternately even and odd, and there are some close connections between the Chebyshev polynomials of the first and second kinds, see~\cite{Mason2003} for details.
For example, the trigonometric identity
$$\sin((n+1)\theta)-\sin((n-1)\theta)=2\sin(\theta)\cos(n\theta)$$
leads to an identity: $U_n(x)-U_{n-2}(x)=2T_n(x)$.
It is well known that the Chebyshev polynomials form an orthogonal basis over the finite interval $[-1,1]$.

In~\cite{Hetyei}, Hetyei showed that derivative polynomials share some properties with Chebyshev polynomials, and
derivative polynomials are closely related to the face enumerating polynomials of the Chebyshev transforms of the Boolean algebras.
As an illustration, we give an example.
\begin{example}[{\cite[Corollary~8.7]{Hetyei}}]\label{Hetyei05}
The zeros of $P_n(x)$ and $Q_n(x)$ are pure imaginary, have multiplicity $1$, belong to the line segment $[-\mathrm{i},\mathrm{i}]$ and are interlaced.
\end{example}

Note that $P_n(-x)=(-1)^{n+1}P_n(x)$ and ${Q}_n(x)=(-1)^n{Q}_n(-x)$.
Then the polynomials $P_n(x)$ and $Q_n(x)$ are both alternately even and odd.
Below are $P_n(x)$ and $Q_n(x)$ for $n\leqslant 5$:
\begin{equation*}
\begin{split}
P_0(x)&=x,~P_1(x)=1+x^2,~P_2(x)=2x+2x^3,~P_3(x)=2+8x^2+6x^4,\\
P_4(x)&=16x+40x^3+24x^5,~P_5(x)=16+136x^2+240x^4+120x^6,\\
{Q}_0(x)&=1,~{Q}_1(x)=x,~{Q}_2(x)=1+2x^2,~{Q}_3(x)=5x+6x^3,\\
{Q}_4(x)&=5+28x^2+24x^4,~{Q}_5(x)=61x+180x^3+120x^5.
\end{split}
\end{equation*}

Motivated by Proposition~\ref{Hetyei05}, we shall explore the following problem.
\begin{problem}\label{Problem2}
If we choose the derivative polynomials for tangent or secant as basis sets of polynomial space, whether there are any interesting findings?
\end{problem}

In the next section, we present various results concerning Problems~\ref{Problem1} and~\ref{Problem2}. In Section~\ref{Section03},
we first introduce the definition of alternating derivative polynomials, and we then present certain combinatorial expansions of these polynomials.
\section{Recurrences for derivative polynomials and Euler numbers}\label{Section2}

Let $P(x;z)$ and ${Q}(x;z)$ be given by~\eqref{EGF}.
Hoffman~\cite[Eq.~(5),~Theorem~3.2]{Hoffman95} found that
\begin{equation*}
\begin{split}
\frac{\mathrm{d}}{\mathrm{d}z}P(x;z)&=1+P^2(x;z),  \\
\frac{\mathrm{d}}{\mathrm{d}z}P(x;z)&=(1+x^2){Q}^2(x;z),\\
\frac{\mathrm{d}}{\mathrm{d}z}{Q}(x;z)&=P(x;z){Q}(x;z),
\end{split}
\end{equation*}
which yield three convolution formulas:
\begin{equation*}
\begin{split}
P_{n+1}(x)&=\delta_{0n}+\sum_{i=0}^n\binom{n}{i}P_i(x)P_{n-i}(x), \\
P_{n+1}(x)&=(1+x^2)\sum_{i=0}^n\binom{n}{i}{Q}_i(x){Q}_{n-i}(x),\\
{Q}_{n+1}(x)&=\sum_{i=0}^n\binom{n}{i}P_i(x){Q}_{n-i}(x).
\end{split}
\end{equation*}
Comparing the last two convolution formula with~\eqref{DerivativeRecu}, we get the following result.
\begin{proposition}
We have
\begin{equation*}
\begin{split}
\frac{\mathrm{d}}{\mathrm{d}x}P_n(x)&=\sum_{i=0}^n\binom{n}{i}{Q}_i(x){Q}_{n-i}(x)~\text{for $n\geqslant 0$},\\
\frac{\mathrm{d}}{\mathrm{d}x}{Q}_n(x)&=\frac{1}{1+x^2}\sum_{i=1}^n\binom{n}{i}P_i(x){Q}_{n-i}(x)~\text{for $n\geqslant 1$}.
\end{split}
\end{equation*}
\end{proposition}

Putting $P_{-1}(x)=1$, then the polynomials
$P_{-1}(x),~P_0(x),~P_1(x),~P_2(x),~\ldots,~P_n(x)$
form a basis for polynomials with degree less than or equal to $n+1$, since they have different degrees.
We can now present the following result.
\begin{theorem}\label{mainthm01}
For any $n\geqslant 0$, one has
\begin{equation}\label{eq:01}
Q_{2n}(x)=(-1)^n+\sum_{k=0}^{n-1}\binom{2n}{2k+1}(-1)^kP_{2n-2k-1}(x),
\end{equation}
\begin{equation}\label{eq:02}
Q_{2n+1}(x)=\sum_{k=0}^n\binom{2n+1}{2k+1}(-1)^kP_{2n-2k}(x).
\end{equation}
\end{theorem}
\begin{proof}
By using~\eqref{EGF}, we obtain
\begin{equation*}
\begin{split}
Q(x;z)&=\frac{\sin^2 (z)+\cos^2 (z)}{\cos (z)}\frac{1}{1-x\tan (z)} \\
&=\left(\cos (z)-x\sin (z)+x\sin (z)+\frac{\sin^2 (z)}{\cos (z)}\right)\frac{1}{1-x\tan (z)}\\
&=\frac{\cos (z)-x\sin (z)}{1-x\tan (z)}+\sin (z)\frac{x+\tan (z)}{1-x\tan (z)}.
\end{split}
\end{equation*}
Thus
\begin{equation}\label{Qxz-Pxz}
Q(x;z)=\cos(z) + \sin(z) P(x;z).
\end{equation}
So we get
\begin{align*}
	&\sum_{n=0}^\infty Q_{2n}(x) \frac{z^{2n}}{(2n)!}
	+ \sum_{n=0}^\infty Q_{2n+1}(x) \frac{z^{2n+1}}{(2n+1)!}\\
	&= \sum_{n=0}^\infty  \frac{(-1)^n z^{2n}}{(2n)! }
			+\sum_{k=0}^\infty \frac {(-1)^kz^{2k+1}}{(2k+1)!} \times \left(\sum_{n=1}^\infty \frac{P_{2n-1}(x)}{(2n-1)!} {z^{2n-1}}+\sum_{n=0}^\infty \frac{P_{2n}(x)}{(2n)!} {z^{2n}}
\right)
\end{align*}
Selecting the coefficients of $z^{2n}$ and $z^{2n+1}$, we get
$$\frac{{Q}_{2n}(x)}{(2n)!}=\frac{(-1)^n}{(2n)!}
		+\sum_{k=0}^{n-1}\frac {(-1)^k}{(2k+1)!} \frac{P_{2n-2k-1}(x)}{(2n-2k-1)!},$$
$$\frac{{Q}_{2n+1}(x)}{(2n+1)!}=\sum_{k=0}^n\frac {(-1)^k}{(2k+1)!}
		\frac{P_{2n-2k}(x)}{(2n-2k)!},$$
and so the proof is complete.
\end{proof}
\begin{example}
Now specialize to the case $n=2$ in Theorem~\ref{mainthm01}, we have
\begin{align*}
Q_4(x)&=5+28x^2+24x^4=4P_3(x)-4P_1(x)+1,\\
Q_5(x)&=61x+180x^3+120x^5=5P_4(x)-10P_2(x)+P_0(x).
\end{align*}
\end{example}
\begin{theorem}\label{thm04}
For any $m\geqslant 0$, we have
\begin{equation}\label{eq:03}
\sum_{k=0}^m\binom{2m}{2k}(-1)^kP_{2m-2k}(x)-x\sum_{k=0}^{m-1}\binom{2m}{2k+1}(-1)^kP_{2m-2k-1}(x)=(-1)^mx,
\end{equation}
\begin{equation}\label{eq:04}
x\sum_{k=0}^m\binom{2m+1}{2k+1}(-1)^kP_{2m-2k}(x)-\sum_{k=0}^m\binom{2m+1}{2k}(-1)^kP_{2m-2k+1}(x)=(-1)^{m+1}.
\end{equation}
\end{theorem}
\begin{proof}
$(i)$\quad Combining~\eqref{DerivativeRecu} and~\eqref{eq:01}, we obtain
\begin{equation*}
\begin{split}
Q_{2m+1}(x)&=(1+x^2)\frac{\mathrm{d}}{\mathrm{d}x}Q_{2m}(x)+xQ_{2m}(x)\\
&=\sum_{k=0}^{m-1}\binom{2m}{2k+1}(-1)^k(1+x^2)\frac{\mathrm{d}}{\mathrm{d}x}P_{2m-2k-1}(x)+xQ_{2m}(x)\\
&=\sum_{k=0}^{m-1}\binom{2m}{2k+1}(-1)^kP_{2m-2k}(x)+xQ_{2m}(x)\\
&=\sum_{k=0}^{m-1}\left(\binom{2m+1}{2k+1}-\binom{2m}{2k}\right)(-1)^kP_{2m-2k}(x)+xQ_{2m}(x)\\
&=\sum_{k=0}^{m-1}\binom{2m+1}{2k+1}(-1)^kP_{2m-2k}(x)-\sum_{k=0}^{m-1}\binom{2m}{2k}(-1)^kP_{2m-2k}(x)+xQ_{2m}(x)\\
&=Q_{2m+1}(x)-(-1)^mP_0(x)-\sum_{k=0}^{m-1}\binom{2m}{2k}(-1)^kP_{2m-2k}(x)+xQ_{2m}(x)\\
&=Q_{2m+1}(x)-\sum_{k=0}^{m}\binom{2m}{2k}(-1)^kP_{2m-2k}(x)+xQ_{2m}(x).
\end{split}
\end{equation*}
Hence
$$xQ_{2m}(x)-\sum_{k=0}^{m}\binom{2m}{2k}(-1)^kP_{2m-2k}(x)=0.$$
It follows from~\eqref{eq:01} that
$$(-1)^mx+x\sum_{k=0}^{m-1}\binom{2m}{2k+1}(-1)^kP_{2m-2k-1}(x)-\sum_{k=0}^{m}\binom{2m}{2k}(-1)^kP_{2m-2k}(x)=0,$$
and so we obtain~\eqref{eq:03}.

$(ii)$\quad Combining~\eqref{DerivativeRecu} and~\eqref{eq:02}, we get
\begin{equation*}
\begin{split}
Q_{2m+2}(x)&=(1+x^2)\frac{\mathrm{d}}{\mathrm{d}x}Q_{2m+1}(x)+xQ_{2m+1}(x)\\
&=\sum_{k=0}^m\binom{2m+1}{2k+1}(-1)^k\left((1+x^2)\frac{\mathrm{d}}{\mathrm{d}x}P_{2m-2k}(x)+xP_{2m-2k}(x)\right)\\
&=\sum_{k=0}^m\binom{2m+1}{2k+1}(-1)^k\left(P_{2m-2k+1}(x)+xP_{2m-2k}(x)\right).
\end{split}
\end{equation*}
Since $$\binom{2m+2}{2k+1}=\binom{2m+1}{2k+1}+\binom{2m+1}{2k},$$
we have
\begin{equation*}
\begin{split}
&Q_{2m+2}(x)\\
&=\sum_{k=0}^m\left(\binom{2m+2}{2k+1}-\binom{2m+1}{2k}\right)(-1)^kP_{2m-2k+1}(x)+x\sum_{k=0}^m\binom{2m+1}{2k+1}(-1)^kP_{2m-2k}(x)\\
&=\sum_{k=0}^m\binom{2m+2}{2k+1}(-1)^kP_{2m-2k+1}(x)+S_m(x),
\end{split}
\end{equation*}
where $$S_m(x)=x\sum_{k=0}^m\binom{2m+1}{2k+1}(-1)^kP_{2m-2k}(x)-\sum_{k=0}^m\binom{2m+1}{2k}(-1)^kP_{2m-2k+1}(x).$$
In view of~\eqref{eq:01}, we find $S_m(x)=(-1)^{m+1}$. This completes the proof
\end{proof}
Recall that $P_{2m-2k+1}(0)=E_{2m-2k+1}$ for any $m\geqslant 0$.
Now specialize to the case $x=0$ in~\eqref{eq:04}, one can get the following result.
\begin{corollary}
For any $m\geqslant 0$, we have
$$\sum_{k=0}^m\binom{2m+1}{2k}(-1)^kE_{2m-2k+1}=(-1)^m.$$
\end{corollary}
Combining~\eqref{Pn1Qn1} and Theorem~\ref{thm04}, one can derive the following result.
\begin{corollary}
For any $m\geqslant 0$, we have
$$\sum_{k=0}^m\binom{2m}{2k}(-1)^k2^{2m-2k}E_{2m-2k}-\sum_{k=0}^{m-1}\binom{2m}{2k+1}(-1)^k2^{2m-2k-1}E_{2m-2k-1}=(-1)^m,$$
$$\sum_{k=0}^m\binom{2m+1}{2k+1}(-1)^k2^{2m-2k}E_{2m-2k}-\sum_{k=0}^m\binom{2m+1}{2k}(-1)^k2^{2m-2k+1}E_{2m-2k+1}=(-1)^{m+1}.$$
\end{corollary}

The {\it Bernoulli numbers} $B_n$ can defined by the exponential generating function
$$\frac{z}{\mathrm{e}^z-1}=\sum_{n=0}^\infty B_n\frac{z^n}{n!}=1-\frac{z}{2}+\frac{1}{6}\frac{z^2}{2!}-\frac{1}{30}\frac{z^4}{4!}+\frac{1}{42}\frac{z^6}{6!}-\frac{1}{30}\frac{z^8}{8!}+\cdots.$$
In particular, $B_{2n+1}=0$ for $n\geqslant 1$, since $\frac{z}{2}\coth\left(\frac{z}{2}\right)$ is an even function and $$\frac{z}{2}\coth\left(\frac{z}{2}\right)=\frac{z}{2}\frac{\mathrm{e}^{z/2}+
\mathrm{e}^{-z/2}}{\mathrm{e}^{z/2}-\mathrm{e}^{-z/2}}=\frac{z}{\mathrm{e}^z-1}+\frac{z}{2}.$$
It is well known that the Bernoulli number has an explicit formula:
$$B_n=\sum_{k=0}^{n}\frac{(-1)^kk!}{k+1}\Stirling{n}{k},$$
where $\Stirling{n}{k}$ is the Stirling number of the second kind,
which counts the number of partitions
of $[n]=\{1,2,\ldots,n\}$ into $k$ nonempty blocks. The Bernoulli numbers appear often in the coefficients
of trigonometric functions (see~\cite[Chapter~6]{Graham} for details).
For example,
$$z\csc(z)=1+\sum_{n=1}^\infty (-1)^{n+1}(4^n-2)B_{2n}\frac{z^{2n}}{(2n)!},~z\cot(z)=\sum_{n=0}^\infty (-4)^nB_{2n}\frac{z^{2n}}{(2n)!} ~\quad~~(0<|z|<\pi).$$

We can now present a result in the reverse direction of Theorem~\ref{mainthm01}.
\begin{theorem}\label{mainthm02}
For any $n\geqslant 0$, one has
\begin{align*}
(2n+1)P_{2n}(x)&=Q_{2n+1}(x)+\sum_{i=0}^{n-1}\binom{2n+1}{2i+2}(-1)^{i+2}(4^{i+1}-2)B_{2i+2}Q_{2n-2i-1}(x),\\
(2n+2)P_{2n+1}(x)&=Q_{2n+2}(x)+\sum_{i=0}^n\binom{2n+2}{2i+2}(-1)^{i+2}(4^{i+1}-2)B_{2i+2}Q_{2n-2i}(x)-(-4)^{n+1}B_{2n+2}.
\end{align*}
\end{theorem}
\begin{proof}
From~\eqref{Qxz-Pxz}, we see that $P(x;z)=\csc (z) Q(x;z) -\cot (z)$. Thus
\begin{equation*}\label{Qxz-Pxz02}
zP(x;z)=z\csc (z) Q(x;z) -z\cot (z).
\end{equation*}
Then we have
\begin{align*}
	&\sum_{n=0}^\infty P_{2n}(x) \frac{z^{2n+1}}{(2n)!}
	+ \sum_{n=0}^\infty P_{2n+1}(x) \frac{z^{2n+2}}{(2n+1)!}\\
	&=\left(1+\sum_{i=0}^\infty (-1)^{i+2}(4^{i+1}-2)B_{2i+2}\frac{z^{2i+2}}{(2i+2)!}\right)\left(\sum_{n=0}^\infty Q_{2n}(x) \frac{z^{2n}}{(2n)!}
	+ \sum_{n=0}^\infty Q_{2n+1}(x) \frac{z^{2n+1}}{(2n+1)!}\right)-\\
&\sum_{n=0}^\infty (-4)^nB_{2n}\frac{z^{2n}}{(2n)!}.
\end{align*}
Equating the coefficients of $z^{2n+1}$ and $z^{2n+2}$, we get
$$\frac{{P}_{2n}(x)}{(2n)!}=\frac{Q_{2n+1}(x)}{(2n+1)!}
		+\sum_{i=0}^{n-1}\frac {(-1)^{i+2}(4^{i+1}-2)B_{2i+2}}{(2i+2)!} \frac{Q_{2n-2i-1}(x)}{(2n-2i-1)!},$$
$$		\frac{{P}_{2n+1}(x)}{(2n+1)!}=\frac{Q_{2n+2}(x)}{(2n+2)!}
+\sum_{i=0}^{n}\frac {(-1)^{i+2}(4^{i+1}-2)B_{2i+2}}{(2i+2)!} \frac{Q_{2n-2i}(x)}{(2n-2i)!}-\frac{(-4)^{n+1}B_{2n+2}}{(2n+2)!},$$
and so the proof is complete.
\end{proof}
\begin{example}
Now specialize to the case $n=2$ in Theorem~\ref{mainthm02}, we have
\begin{align*}
&Q_{5}(x)+\sum_{i=0}^{1}\binom{5}{2i+2}(-1)^{i+2}(4^{i+1}-2)B_{2i+2}Q_{3-2i}(x)\\
&=Q_{5}(x)+10(4-2)B_2Q_3(x)-5(4^2-2)B_4Q_1(x)\\
&=61x+180x^3+120x^5+\frac{20}{6}(5x+6x^3)+\frac{7}{3}x\\
&=5P_4(x),\\
&Q_{6}(x)+\sum_{i=0}^2\binom{6}{2i+2}(-1)^{i+2}(4^{i+1}-2)B_{2i+2}Q_{4-2i}(x)-(-4)^{3}B_{6}\\
&=Q_{6}(x)+30B_2Q_4(x)-210B_4Q_2(x)+62B_6+64B_6\\
&=61+662x^2+1320x^4+720x^6+5(5+28x^2+24x^4)+7(1+2x^2)+3\\
&=6P_5(x).
\end{align*}
\end{example}

Taking $x=1$ in Theorem~\ref{mainthm02} leads to the following result,
which gives connections among tangent numbers, secant numbers, Springer numbers and Bernoulli numbers.
\begin{corollary}
For any $n\geqslant 0$, one has
\begin{align*}
(2n+1)2^{2n}E_{2n}&=s_{2n+1}+\sum_{i=0}^{n-1}\binom{2n+1}{2i+2}(-1)^{i+2}(4^{i+1}-2)B_{2i+2}s_{2n-2i-1},\\
(2n+2)2^{2n+1}E_{2n+1}&=s_{2n+2}+\sum_{i=0}^n\binom{2n+2}{2i+2}(-1)^{i+2}(4^{i+1}-2)B_{2i+2}s_{2n-2i}-(-4)^{n+1}B_{2n+2}.
\end{align*}
\end{corollary}
Putting $x=0$ in Theorem~\ref{mainthm02} leads to the following result.
\begin{corollary}
For any $n\geqslant 0$, one has
$$E_{2n+2}=(2n+2)E_{2n+1}-\sum_{i=0}^n\binom{2n+2}{2i+2}(-1)^{i+2}(4^{i+1}-2)B_{2i+2}E_{2n-2i}+(-4)^{n+1}B_{2n+2}.$$
\end{corollary}
\section{Eulerian polynomials and alternating derivative polynomials}\label{Section03}
Let $\msn$ denote the symmetric group of all permutations of $[n]$.
A {\it descent} of $\pi\in\msn$ is an index $i\in [n-1]$ such that $\pi(i)>\pi(i+1)$.
Denote by $\des(\pi)$ the number of descents of $\pi$.
The {\it Eulerian polynomials of type $A$} are defined by
$$A_n(x)=\sum_{\pi\in\msn}x^{\des(\pi)+1}.$$
Let $\pm[n]=[n]\cup\{\overline{1},\ldots,\overline{n}\}$, where $\overline{i}=-i$.
Denote by $\mbn$ the hyperoctahedral group of rank $n$. Elements of $\mbn$ are signed
permutations $\sigma$ of the set $\pm[n]$ such that $\sigma\left(-i\right)=-\sigma(i)$ for all $i$.
Let $\sigma=\sigma(1)\sigma(2)\cdots\sigma(n)\in \mbn$.
The {\it type $B$ Eulerian polynomials} are defined by
\begin{equation*}
B_n(x)=\sum_{\sigma\in\mbn}x^{\des_B(\sigma)},
\end{equation*}
where $\des_B(\sigma)=\#\{i\in [0,n-1]:~\sigma(i)>\sigma({i+1}),~\sigma(0)=0\}$.
For $n\geqslant 1$, they satisfy the following recurrence relations (see~\cite{Brenti94,Chow08,Petersen15} for instance):
 \begin{equation}\label{ABnx-recu}
	\left\{
	\begin{array}{l}
A_{n}(x)=nxA_{n-1}(x)+x(1-x)\frac{\mathrm{d}}{\mathrm{d}x}A_{n-1}(x),~A_0(x)=1\\
B_{n}(x)=(2nx+1-x)B_{n-1}(x)+2x(1-x)\frac{\mathrm{d}}{\mathrm{d}x}B_{n-1}(x),~B_0(x)=1.
	\end{array}
	\right.
	\end{equation}

The derivative polynomials for hyperbolic
tangent and secant are defined by
\begin{equation*}\label{derivapoly-2}
\frac{d^n}{d\theta^n}\tanh(\theta)=\widetilde{P}_n(\tanh(\theta))\quad {\text and}\quad
\frac{d^n}{d\theta^n}\operatorname{sech}(\theta)=\operatorname{sech}(\theta)\widetilde{Q}_n(\tanh(\theta)).
\end{equation*}
Let $\mathrm{i}=\sqrt{-1}$.
Since $\tanh(\theta)=-\mathrm{i}\tan(\mathrm{i}\theta)$ and $\operatorname{sech}(\theta)=\sec(\mathrm{i}\theta)$, we have
\begin{equation*}\label{pnxqnx}
\widetilde{P}_n(x)=\mathrm i^{n-1}P_n(\mathrm i x) \quad
{\text and}\quad \widetilde{Q}_n(x)=\mathrm i^{n}Q_n(\mathrm i x).
\end{equation*}
Using~\eqref{DerivativeRecu}, we obtain that
 \begin{equation}\label{pqnx-recu02}
	\left\{
	\begin{array}{l}
\widetilde{P}_{n+1}(x)=(1-x^2)\frac{\mathrm{d}}{\mathrm{d}x}\widetilde{P}_n(x),~\widetilde{P}_0(x)=x;\\
\widetilde{Q}_{n+1}(x)=(1-x^2)\frac{\mathrm{d}}{\mathrm{d}x}\widetilde{Q}_n(x)-x\widetilde{Q}_n(x),~\widetilde{Q}_0(x)=1.
	\end{array}
	\right.
	\end{equation}
\begin{definition}
The alternatingly derivative polynomials for tangent and secant are respectively defined as follows:
$$p_n(x)=(-1)^n\widetilde{P}_{n}(x)=(-1)^n\mathrm{i}^{n-1}P_n(\mathrm{i}x)~~{\text{for $n\geqslant 1$,~$p_0(x)=x$}},$$
$$q_n(x)=(-1)^n2^n\widetilde{Q}_{n}(x)=(-1)^n2^n\mathrm i^{n}Q_n(\mathrm i x)~~{\text{for $n\geqslant 0$,~$q_0(x)=1$}}.$$
\end{definition}
Let
\begin{equation*}
P_n(x)=\sum_{k=0}^{\lfloor(n+1)/2\rfloor}p(n,n-2k+1)x^{n-2k+1},~Q_n(x)=\sum_{k=0}^{\lfloor {n/2}\rfloor}q(n,n-2k)x^{n-2k}.
\end{equation*}
It is easy to verify that
\begin{equation*}
p_n(x)=\sum_{k=0}^{\lfloor(n+1)/2\rfloor}p(n,n-2k+1)(-1)^kx^{n-2k+1},~q_n(x)=2^n\sum_{k=0}^{\lfloor {n/2}\rfloor}q(n,n-2k)(-1)^kx^{n-2k}.
\end{equation*}

Using~\eqref{pqnx-recu02}, we get
 \begin{equation}\label{abnx-recu}
	\left\{
	\begin{array}{l}
p_{n+1}(x)=(x^2-1)\frac{\mathrm{d}}{\mathrm{d}x}{p}_n(x),~p_0(x)=x;\\
q_{n+1}(x)=2(x^2-1)\frac{\mathrm{d}}{\mathrm{d}x}q_n(x)+2xq_n(x),~q_0(x)=1.
	\end{array}
	\right.
	\end{equation}
It should be noted that the distribution of zeros of $p_n(x)$ and $q_n(x)$ has been given by Hetyei~\cite[Theorem~8.6]{Hetyei}.
Comparing~\eqref{ABnx-recu} and~\eqref{abnx-recu}, it is routine to verify the following result,
which is equivalent to~\cite[Theorem~5,~Theorem~6]{Franssens07}. For the sake of brevity we omit the proof.
\begin{theorem}\label{propo-pq}
For $n\geqslant 1$, we have
\begin{equation*}\label{anxbnx-def}
p_n(x)=(x+1)^{n+1}A_{n}\left(\frac{x-1}{x+1}\right),\quad
q_n(x)=(x+1)^nB_n\left(\frac{x-1}{x+1}\right).
\end{equation*}
\end{theorem}
In the past decades, Eulerian polynomials have been extensively studied, see~\cite{Petersen15} and references therein.
Using Theorem~\ref{propo-pq}, one can derive several properties of alternating derivative polynomials,
including generating functions, particular values, explicit formulas and combinatorial expansions.
In the sequel we shall give two applications of Theorem~\ref{propo-pq}.

The most famous expansion of Eulerian polynomials is the {\it Frobenius formula}:
\begin{equation}\label{Frobenius01}
A_n(x)=x\sum_{k=0}^nk!\Stirling{n}{k}(x-1)^{n-k}.
\end{equation}
Brenti~\cite[Theorem~3.14]{Brenti94} found a generalization of~\eqref{Frobenius01}:
\begin{equation}\label{Frobenius02}
B_n(x)=\sum_{k=0}^nk!\sum_{i=k}^n\binom{n}{i}2^i\Stirling{i}{k}(x-1)^{n-k}.
\end{equation}
Therefore, the following result follows from~\eqref{Frobenius01},~\eqref{Frobenius02} and Theorem~\ref{propo-pq}.
\begin{corollary}
For $n\geqslant 1$, we have
\begin{align*}
p_n(x)&=(x-1)\sum_{k=0}^n(-2)^{n-k}k!\Stirling{n}{k}(x+1)^k,\\
q_n(x)&=\sum_{k=0}^n(-2)^{n-k}k!\sum_{i=k}^n\binom{n}{i}2^i\Stirling{i}{k}(x+1)^{k}.
\end{align*}
\end{corollary}

Let $f(x)=\sum_{i=0}^nf_ix^i$ be a symmetric polynomial, i.e., $f_i=f_{n-i}$ for any $0\leqslant i\leqslant n$. Then $f(x)$ can be expanded uniquely as
$f(x)=\sum_{k=0}^{\lrf{\frac{n}{2}}}\gamma_kx^k(1+x)^{n-2k}$, and it is said to be {\it $\gamma$-positive}
if $\gamma_k\geqslant 0$ for $0\leqslant k\leqslant \lrf{\frac{n}{2}}$.
The $\gamma$-positivity provides a natural approach to study symmetric and unimodal polynomials,
see~\cite{Athanasiadis17} and references therein.

Let $\pi\in\msn$.
An index $i\in [n]$ is a {\it peak} (resp.~{\it double descent})
of $\pi$ if $\pi(i-1)<\pi(i)>\pi(i+1)$ (resp. $\pi(i-1)>\pi(i)>\pi(i+1)$), where $\pi(0)=\pi(n+1)=0$.
Let $a(n,k)$ be the number of permutations in $\msn$ with $k$ peaks and without double descents.
Foata and Sch\"utzenberger~\cite{Foata70} discovered that
\begin{equation}\label{Anx-gamma}
A_n(x)=\sum_{k=1}^{\lrf{({n+1})/{2}}}a(n,k)x^k(1+x)^{n+1-2k}.
\end{equation}
Following~\cite{Chow08,Petersen15}, one has
\begin{equation}\label{Bnx-gamma}
B_n(x)=\sum_{k=0}^{\lrf{{n}/{2}}}4^kb(n,k)x^k(1+x)^{n-2k},
\end{equation}
where $b(n,k)=\#\{\pi\in\msn:~~i\in[n-1],~\pi(i-1)<\pi(i)>\pi(i+1),~\pi(0)=0\}$.
Substituting~\eqref{Anx-gamma} and~\eqref{Bnx-gamma} into Theorem~\ref{propo-pq},
we deduce the following.
\begin{corollary}\label{pnqn}
For $n\geqslant 1$, we have
$$p_n(x)=\sum_{k=1}^{\lrf{({n+1})/{2}}}a(n,k)(x^2-1)^{k}(2x)^{n+1-2k},$$
$$q_n(x)=\sum_{k=0}^{\lrf{{n}/{2}}}4^kb(n,k)(x^2-1)^k(2x)^{n-2k}.$$
\end{corollary}
The Chebyshev polynomials can be defined in terms of the sums (see~\cite{Mason2003} for instance):
\begin{equation}\label{TU}
T_n(x)=\sum_{k=0}^{\lrf{n/2}}\binom{n}{2k}(x^2-1)^kx^{n-2k},~
U_n(x)=\sum_{k=0}^{\lrf{n/2}}\binom{n+1}{2k+1}(x^2-1)^kx^{n-2k}.
\end{equation}
Therefore, Corollary~\ref{pnqn} can be seen as a dual of~\eqref{TU},
which indicate that the alternating derivative polynomials deserve more work.
\section*{Acknowledgements.}
This work is supported by NSFC (12071063).

\end{document}